\documentclass[11pt]{amsart}

\usepackage{amssymb,amsmath,amsthm,amsfonts}
\usepackage{graphicx}
\usepackage{bbm}
\usepackage{hyperref}
\usepackage{tikz}
\usepackage{tikz-cd}
\usepackage[utf8]{inputenc}
\usepackage{mathtools}
\usepackage{yhmath}
\usepackage{theoremref}
\usepackage{cite}
\usepackage{comment}
\usepackage[margin=1in]{geometry}
\usepackage{enumerate}

\newtheorem{introthm}{Theorem}

\newtheorem{thm}{Theorem}[section]
\newtheorem{lemma}[thm]{Lemma}

\theoremstyle{definition}

\theoremstyle{remark}

\DeclareMathOperator{\re}{Re}

\DeclareMathOperator{\id}{id}

\DeclareMathOperator{\Aut}{Aut}

\DeclareMathOperator{\alg}{alg}

\newcommand{\twoone}{II$_1$ }

\newcommand{\cM}{\mathcal{M}}

\newcommand{\cU}{\mathcal{U}}

\newcommand{\N}{\mathbb{N}}

\newcommand{\R}{\mathbb{R}}

\newcommand{\Z}{\mathbb{Z}}

\newcommand{\fc}{\mathfrak{c}}

\DeclarePairedDelimiter{\norm}{\lVert}{\rVert}
\DeclarePairedDelimiter{\ip}{\langle}{\rangle}

\begin{document}

\title{Elementary embeddings into ultrapower $\mathrm{II}_1$ factors without a UCP lift}

\author{David Gao}
\address{Department of Mathematical Sciences, UCSD, 9500 Gilman Dr, La Jolla, CA 92092, USA}
\email{weg002@ucsd.edu}
\urladdr{https://sites.google.com/ucsd.edu/david-gao/home}

\author{David Jekel}
\address{Department of Mathematical Sciences, University of Copenhagen, Universitetsparken 5, 2100 Copenhagen {\O}, Denmark}
\email{daj@math.ku.dk}
\urladdr{https://davidjekel.com}

\maketitle

\begin{abstract}
We show that there are $\mathrm{II}_1$ factors $M$ and elementary embeddings $M \to M^{\cU}$ which do not lift to sequences of UCP maps, and in fact $M$ can be chosen from any given elementary equivalence class.  Furthermore, under continuum hypothesis, we show that in the sense of cardinality ``most'' automorphisms of a ultrapower $M^{\cU}$ of a separable $\mathrm{II}_1$ factor do not lift to a sequence of UCP maps $\varphi_n: M \to M$.
\end{abstract}

\section{Introduction}

The significance of ultrapowers in von Neumann algebras has been recognized since McDuff's use of central sequence algebras in \cite{McD69a,McD69b} and was a large motivation for the development of continuous model theory for von Neumann algebras \cite{FHS2014a,FHS2014a}.  Ultrapowers are a powerful tool because many phenomena that happen approximately on $M$ can happen exactly in $M^{\cU}$, and the exact versions are much easier to manipulate.  Conversely, what happens in $M^{\cU}$ must happen approximately in $M$, as every element in $M^{\cU}$ lifts to a sequence in $M$, but it is a delicate and important question how much of the $*$-algebra structure can be preserved by the lifted sequences.  For instance, the recent approach to classification of $*$-homomorphisms between $\mathrm{C}^*$-algebras relies on lifting $*$-homomorphisms into ultraproducts or sequence algebras to $*$-homomorphisms (not merely asymptotic $*$-homomorphisms) into the algebra itself (see \cite[\S 6]{classification}).

It is well-known that a $*$-homomorphism from the hyperfinite $\mathrm{II}_1$ factor $R$ into an ultraproduct of $\mathrm{II}_1$ factors lifts to a sequence of $*$-homomorphisms, that is, if $\varphi: R \to \prod_{n \to \cU} M_n$ is a $*$-homomorphism, there exist $*$-homomorphisms $\varphi_i: R \to M_n$ such that $\varphi(x) = [\varphi_n(x)]_n$ for all $x \in R$.  Atkinson and Kunnawalkam Elayavalli \cite[Theorem 3.5]{AKE2021} showed conversely that if $M$ is a $\mathrm{II}_1$ factor that is Connes-embeddable (i.e.\ it embeds into $R^{\cU}$) and if every $*$-homomorphism $M \to M^{\cU}$ lifts to a sequence of $*$-homomorphisms, then $M \cong R$.  The same is true if the lifting is allowed to be a general UCP map rather than a $*$-homomorphism \cite[Theorem 3.7]{AKE2021}.

The question of whether an embedding $M \to \prod_{n \to \cU} M_n$ lifts is closely related to the question of classifying such embeddings up to various kinds of equivalence.  Indeed, the reason that embeddings of $R$ into ultraproducts of $\mathrm{II}_1$ factors lift is that any two embeddings of $R$ into $\prod_{n \to \cU} M_n$ are unitarily conjugate, so in particular any embedding is unitarily conjugate to a diagonal embedding, which lifts.  Jung \cite{Jun07} showed the remarkable converse that if any two embeddings of a Connes-embeddable $\mathrm{II}_1$ factor $M$ into $R^{\cU}$ are unitarily conjugate, then $M \cong R$.  In fact, Ozawa showed that in fact, for a Connes-embeddable $\mathrm{II}_1$ factor $M$, separability of the space of embeddings into $R^{\cU}$ up to unitary equivalence is enough to conclude that $M \cong R$.  Atkinson and Kunnawalkam Elayavalli \cite[Theorem 4.7]{AKE2021} strengthened Jung's result by replacing unitary conjugacy with UCP conjugacy; they showed that $M \cong R$ provided that $M$ is Connes-embeddable and for any two embeddings $\alpha, \beta: M \to R^{\cU}$ we have $\beta = \varphi \circ \alpha$ where $\varphi: R^{\cU} \to R^{\cU}$ is a linear map of the form $\varphi([x_n]_{n \in \N}) = [\varphi_n(x_n)]_{n \in \N}$ for unital completely positive maps $\varphi_n: R \to R$.

Moreover, Atkinson, Goldbring, and Kunnawalkam Elayavalli \cite[Theorem 3.2.1]{AGKE22} showed that if $M$ is Connes-embeddable and any two embeddings of $M$ into $M^{\cU}$ are conjugate by an automorphism of $M^{\cU}$, then $M \cong R$.  The classification of embeddings up to automorphic conjugacy (as opposed to unitary conjugacy for instance) is inherently model-theoretic, since two embeddings $\alpha, \beta: M \to N^{\cU}$ are automorphically conjugate if and only if the images of chosen generators have the same type in $N^{\cU}$ (provided that $M$ and $N$ are separable, $\cU$ is an ultrafilter on $\N$, and we assume the continuum hypothesis; see \cite[Fact 2.2.2]{AGKE22} or \cite[Corollary 16.6.5]{Farah2019}).

Our goal in this note is to illustrate the difference between automorphic conjugacy and UCP conjugacy (hence also unitary conjugacy), answering a question posed by Kunnawalkam Elayavalli in private correspondence.  We first give fairly explicit examples of elementary embeddings (in the sense of continuous model theory) that do not admit UCP lifts.  Our construction (see \S \ref{sec: explicit construction}) leverages the rigidity of property (T) groups, similarly to many other counterexamples for related questions (e.g.\ \cite{OzawaUnivSep}, \cite[\S 6]{Brown2011}, \cite[\S 5.5]{GJNS2022}, and \cite{Farah2023}).

\begin{introthm} \label{thm: main 2}
Fix a free ultrafilter $\cU$ on $\N$.  Let $M$ be a separable \twoone factor.  Then there exists a separable \twoone factor $N$ elementarily equivalent to $M$ and an elementary embedding $N \to N^{\cU}$ that does not admit a UCP lifting.
\end{introthm}

Note that under continuum hypothesis (CH), the embedding $\iota$ from Theorem \ref{thm: main 2} is automorphically conjugate to the diagonal embedding $\Delta_N: N \to N^{\cU}$, or $\iota = \alpha \circ \Delta_N$ for some $\alpha \in \Aut(N^{\cU})$, since any two elementary embeddings have the same type.  However, they cannot be UCP conjugate; indeed, if we had $\iota = \varphi \circ \Delta$ where $\varphi([x_n]_{n \in \N}) = [\varphi_n(x_n)]_{n \in \N}$ for some UCP maps $\varphi_n: N \to N$, that would mean that $\iota(x) = [\varphi_n(x)]_{n \in \N}$ for $x \in N$ contradicting the theorem.  In particular, $\iota$ and $\Delta_N$ are also not unitarily conjugate.

Of course, we cannot always take $N = M$ in Theorem \ref{thm: main 2}, because in the case that $M = R$,  every elementary embedding $R \to R^{\cU}$ does lift.  It would be an interesting question to determine whether $N$ can be taken equal to $M$ for general non-amenable $M$.

Theorem \ref{thm: main 2} likewise shows that under CH there is an automorphism $\alpha: N^{\cU} \to N^{\cU}$ that does not lift to a sequence of UCP maps, in the sense that we cannot write $\alpha([x_n]_{n \in \N}) = [\varphi_n(x_n)]_{n \in \N}$.  Our next result shows that, in the sense of cardinality, ``most'' automorphisms of an ultraproduct of tracial von Neumann algebras do not admit UCP liftings.

\begin{introthm} \label{thm: main theorem}
Fix a free ultrafilter $\cU$ on $\N$. Let $M_n$ be a sequence of separable tracial von Neumann algebras such that the ultraproduct $\cM = \prod_{n \to \cU} M_n$ is not completely atomic.  Then the set of automorphisms of $\cM$ that admit UCP liftings (or even liftings to uniformly bounded sequences of positive maps) has cardinality $\fc$.  In particular, assuming continuum hypothesis, there are $2^\fc$ many automorphisms that do not admit UCP liftings.
\end{introthm}

Here the computation of the cardinality of $\Aut(\mathcal{M})$ under continuum hypothesis is a well-known result in model theory (see \cite[Theorems 16.4.1 and 16.6.3]{Farah2019}).  A similar cardinality argument was used in \cite[Corollary 16.7.2]{Farah2019} to show that automorphisms of $\mathrm{C}^*$-ultraproducts (or metric structures more generally) do not necessarily lift to sequences of automorphisms.  The novel point in our result is how to upper bound the cardinality of the set of maps that do have a lifting.  The positive maps in the lifting need not be normal, and hence may be discontinuous with respect to the $2$-norm metric.  But we show that if a positive lifting exists, then the maps $\varphi_n$ in the lifting can be modified to be normal (see Lemma \ref{lem: normal-pos}), and hence determined by their values on a countable set.

\subsection*{Acknowledgements}

We thank Srivatsav Kunnawalkam Elayavalli for discussions that motivated this work.

Gao was supported by National Science Foundation (US), grant DMS-2451697. Jekel was supported by a Marie Sk{\l}odowska-Curie Action from the European Union (FREEINFOGEOM, grant 101209517).  Views and opinions expressed are those of the author(s) only and do not necessarily reflect those of the European Union or the Research Executive Agency. Neither the European Union nor the granting authority can be held responsible for them.

\section{Preliminaries}

We assume familiarity with $\mathrm{C}^*$-algebras, tracial von Neumann algebras, $*$-homomorphisms, and positive, completely positive, and normal maps (see e.g.\ \cite{Dix81,BrownOzawa2008,Blackadar2006}).  In particular, we recall that if $M \subseteq N$ is an inclusion of tracial von Neumann algebras, then there is always a unique trace-preserving conditional expectation $E_M: N \to M$.

We use several notions from the model theory of metric structures (for background, see \cite{BYBHU08,Hart2023}), although the results needed here can mostly be taken as black boxes.  We recall briefly that a language $\mathcal{L}$ for metric structures is given by specifying a collection of sorts, a collection of operations or functional symbols with specified domain and range sorts, and a collection of real-valued predicate symbols with specified domain sorts (the function and predicate symbols also come with specified moduli of continuity).  For instance, in the language of tracial von Neumann algebras (see \cite{FHS2014a,GH2023}), the function symbols signify the $*$-algebra operations and the predicate symbols signify the real and imaginary parts of the trace and the $L^2$-metric.  An $\mathcal{L}$ structure is a collection of complete metric spaces corresponding to the sorts of $\mathcal{L}$, equipped with operations corresponding to the function symbols, and real-valued functions corresponding to the predicate symbols.

Formulas with a certain set of variables are constructed recursively:  Terms are compositions of function symbols (which for von Neumann algebras would result in $*$-polynomials), basic formulas are predicate symbols composed with terms, and formulas in general are obtained from basic formulas by iteratively applying supremum and infimum quantifiers and continuous functions $\R^k \to \R$.  Each formula $\varphi$ with $m$ variables can be interpreted or evaluated in each structure $M$ as a function $\varphi^M: M^m \to \R$.  While a detailed understanding of formulas is not needed for this work, we will use the following notions:
\begin{itemize}
    \item $M$ and $N$ are \emph{elementarily equivalent} if $\varphi^M = \varphi^N$ for all formulas $\varphi$ with no free variables.
    \item An \emph{elementary embedding} $M \to N$ of metric structures (for instance, of tracial von Neumann algebras) is a map $\iota: M \to N$ such that $\varphi^N(\iota(x_1),\dots,\iota(x_m)) = \varphi^M(x_1,\dots,x_m)$ for all formulas $\varphi$ and $x_1$, \dots, $x_m$ in $M$.  In this case, $M$ is said to be an \emph{elementary substructure} of $N$.
    \item Tuples $(x_i)_{i \in I}$ in $M$ and $(y_i)_{i \in I}$ in $N$ have the \emph{same type} if $\varphi^M(x_i: i \in I) = \varphi^N(y_i: i \in I)$ for all formulas $\varphi$ in variables indexed by $I$.
    \item The \emph{density character} of a metric space is the smallest cardinality of a dense subset.
    \item The language $\mathcal{L}$ is said to be \emph{separable} if the number of sorts is countable and for each sorts $S_1$, \dots, $S_m$, the space of formulas with variables $x_j \in S_j$ for $j = 1$, \dots, $m$ is separable with respect to the uniform norm.  The language of tracial von Neumann algebras is separable (see \cite{FHS2014a}).
    \item The \emph{downward L{\"o}wenheim--Skolem theorem} (see \cite[Proposition 7.3]{BYBHU08}) says that if $M$ is a metric structure in a separable language $\mathcal{L}$ and $A \subseteq M$, then $A$ is contained in some elementary substructure $N$ with density character equal to that of $A$.
\end{itemize}

\section{Proof of Theorem \ref{thm: main 2}} \label{sec: explicit construction}

A discrete group $G$ is said to have \emph{Kazhdan's property (T)} if there exist $g_1$, \dots, $g_m$ in $G$ and a constant $K > 0$ such that for every unitary representation $\pi$ of $G$ on a Hilbert space $H$ and $\xi \in H$,
\[
\norm{\xi -  P_{\operatorname{inv}} \xi}_H \leq K \max_{j=1,\dots,m} \norm{\pi(g_j) \xi - \xi}_H,
\]
where $P_{\operatorname{inv}}$ is the projection onto the subspace of $G$-invariant vectors $H_{\operatorname{inv}} = \{\xi \in H: \pi(g) \xi = \xi \forall g \in G\}$.  We call $g_1$, \dots, $g_m$ a \emph{Kazhdan set} and $K$ the associated \emph{Kazhdan constant}.  Property (T) was defined by Kazhdan in \cite{KazhdanTDef} and has many equivalent formulations; see  \cite{BHVpropertyT} for background.  Kazhdan's work also implies that $SL_3(\mathbb{Z})$ has property (T).

\begin{lemma} \label{lem: property T kernel}
    Let $G$ be an countably infinite group with property (T), with a Kazhdan set $\{g_1, \cdots, g_m\}$ and constant $K$.  Let $\pi_1: G \to G_1$ and $\pi_2: G \to G_2$ be two quotient maps. Let $\varphi: L(G_1) \to L(G_2)$ be a unital completely positive map such that $\max_j \norm{\varphi(\pi_1(g_j)) - \pi_2(g_j)}_2 < 1/8K^2$.
    
    Then $\ker(\pi_1) \subset \ker(\pi_2)$.  Moreover, if $\varphi$ is normal, then $\ker(\pi_1)$ is finite index in $\ker(\pi_2)$.

    %Moreover, if $\Phi$ is a UCPT map such that $\max_j \norm{\Phi(\pi_1(g_j)) - \pi_2(g_j)}_2 < 1/8K^2$, then $\ker(\pi_1) = \ker(\pi_2)$.
\end{lemma}

\begin{proof}
	Let $H$ be the Hilbert-space completion of the $\mathrm{C}^*$-correspondence given by $\varphi$.  More concretely, $H$ is formed as the separation-completion of $L(G_1) \otimes_{\alg} L(G_2)$ with respect to the inner product
	\begin{equation*}
	    \ip{x_1 \otimes y_1, x_2 \otimes y_2} = \tau_{L(G_2)}(y_1^* \varphi(x_1^*x_2) y_2).
	\end{equation*}
    (For background, see \cite[\S 4.6, \S F]{BrownOzawa2008}.)  From the general theory of correspondences, there is a $*$-homomorphism $\lambda: L(G_1) \to B(H)$ given by left multiplication on the left tensorand.  Using the fact that the inner product is given by $\tau_{L(G_2)}$ on the outside, there is anti-$*$-homomorphism $\rho: L(G_2) \to B(H)$ given by right-multiplication on the right-tensorand.  Now consider the group representation $\sigma: G \to U(H)$ given by
	\begin{equation*}
	    \sigma(g) = \lambda(\pi_1(g)) \rho(\pi_2(g^{-1})).
	\end{equation*}
	
	Let $\xi = 1 \otimes 1$ in $H$. Then,
    \begin{equation*}
    \begin{split}
        \norm{\sigma(g) \xi - \xi}_H^2 &= \norm{\lambda(\pi_1(g)) \xi - \rho(\pi_2(g)) \xi}_H^2\\
        &= \norm{\pi_1(g) \otimes 1 - 1 \otimes \pi_2(g)}_H^2\\
        &= \tau_{L(G_2)}(\varphi(\pi_1(g)^*\pi_1(g))) - 2 \re \tau_{L(G_2)}(\varphi(\pi_1(g)^*) \pi_2(g)) + \tau_{L(G_2)}(\pi_2(g)^* \varphi(1) \pi_2(g))\\
        &= 2 - 2 \re \tau_{L(G_2)}(\varphi(\pi_1(g)^*) \pi_2(g))\\
        &= 2 \re \tau_{L(G_2)}((\pi_2(g)^* - \varphi(\pi_1(g))) \pi_2(g))\\
        &\leq 2 \norm{\pi_2(g) - \varphi(\pi_1(g))}_2.
    \end{split}
    \end{equation*}
	
	In particular,
	\begin{equation*}
	    \max_j \norm{\sigma(g_j)\xi - \xi}_H \leq \max_j \sqrt{2 \norm{\pi_2(g_j) - \varphi(\pi_1(g_j))}_2} < \frac{1}{2K}.
	\end{equation*}
	
	Therefore, there exists an invariant vector $\eta$ with $\norm{\xi - \eta}_H < 1/2$.  Hence, by the triangle inequality, we see that
	\begin{equation*} %\label{eq: uniform almost invariance}
	    \sup_{g \in G} \norm{\sigma(g) \xi - \xi}_H < 1.
	\end{equation*}
    
	If $g \in \ker(\pi_1)$, then
	\begin{equation*}
    \begin{split}
        |1 - \tau_{L(G_2)}(\pi_2(g))| &\leq \norm{1 - \pi_2(g)}_2\\
        &= \norm{\xi - \rho(\pi_2(g))\xi}_H\\
        &= \norm{\lambda(\pi_1(g)) \xi - \rho(\pi_2(g)) \xi}_H\\
        &= \norm{\sigma(g) \xi - \xi}_H\\
        &< 1.
    \end{split}
	\end{equation*}
    
	Since $\tau_{L(G_2)}(\pi_2(g))$ is either zero or one, we see that it must be one and so $\pi_2(g) = 1$, or $g \in \ker(\pi_2)$.  Thus, $\ker(\pi_1) \subseteq \ker(\pi_2)$.

    Now suppose that $\varphi$ is normal, or it is continuous with respect to the weak-$*$ topology on the unit ball of $L(G_1)$.  Suppose that for contradiction that $\ker(\pi_1)$ is infinite index in $\ker(\pi_2)$.  Then $G_2$ is a quotient of $G_1$ by a subgroup of infinite index.  Thus, there exists an family $(h_j)_{j \in \N}$ in $G$ such that $\pi_2(h_j) = 1$ and the elements $\pi_1(h_j)$ are distinct.  Hence, $\pi_1(h_j)$ are mutually orthogonal unitary elements in $L(G_1)$ and so $\pi_1(h_j) \to 0$ in the weak-$*$ topology in $L(G_1)$.  Since $\varphi$ is normal, the representation $\lambda: L(G_1) \to B(H)$ is also normal.  Therefore, $\lambda(\pi_1(h_j)) \to 0$ in the weak operator topology, so in particular,
    \[
    \lim_{j \to \infty} \ip{\xi, \lambda(\pi_1(h_j)) \xi} = 0.
    \]
    On the other hand, since $\pi_2(h_j) = 1$, we have for all $j \in \N$ that
    \begin{equation*}
    \begin{split}
        |1 - \ip{\xi, \lambda(\pi_1(h_j)) \xi}| &\leq \norm{\xi - \lambda(\pi_1(h_j))\xi}_H\\
        &= \norm{\rho(\pi_2(h_j)) \xi - \lambda(\pi_1(h_j)) \xi}_H\\
        &= \norm{\xi - \sigma(h_j)\xi}_H\\
        &\leq \sup_{g \in G} \norm{\xi - \sigma(g) \xi}_H \\
        &< 1.
    \end{split}
    \end{equation*}
    This implies that $\re \ip{\xi, \lambda(\pi_1(h_j)) \xi}$ has a strictly positive lower bound, contradicting the fact $\ip{\xi, \lambda(\pi_1(h_j)) \xi}  \to 0$.
    %In case $\Phi$ is trace-preserving, we similarly have, for any $g \in \ker(\pi_2)$,
    %\begin{equation*}
    %\begin{split}
     %   |1 - \tau_{L(G_1)}(\pi_1(g))| &\leq \norm{1 - \pi_1(g)}_2\\
      %  &= \norm{\xi - \lambda(\pi_1(g))\xi}_H\\
       % &= \norm{\rho(\pi_2(g)) \xi - \lambda(\pi_1(g)) \xi}_H\\
       % &= \norm{\xi - \sigma(g)\xi}_H\\
       % &< 1.
    %\end{split}
    %\end{equation*}
    %So, again, we have $g \in \ker(\pi_1)$.
\end{proof}

We also recall the following well-known fact.

\begin{lemma} \label{lem: normal perturbation}
Let $M$ and $N$ be von Neumann algebras with $N$ finite-dimensional. Given $\varepsilon > 0$, a finite $F \subset M$ and a UCP map $\varphi: M \to N$, there exists a normal UCP map $\varphi': M \to N$ such that $\max_{x \in F} \norm{\varphi(x) - \varphi'(x)} < \varepsilon$.
\end{lemma}

\begin{proof}
We first prove the claim with $\varphi$ and $\varphi'$ completely positive maps that are not necessarily unital.  Recall that $N$ is a direct sum of matrix algebras $\mathbb{M}_{n_i}$.  Letting $\varphi_i$ be the projection of $\varphi$ onto $\mathbb{M}_{n_i}$, it suffices to show that for each $i$, there exists a normal CP map $\varphi_i'$ with $\max_{x \in F} \norm{\varphi_i(x) - \varphi_i'(x)} < \varepsilon$.

By \cite[Proposition 1.5.14]{BrownOzawa2008}, completely positive maps $M \to \mathbb{M}_{n_i}$ are in bijection with positive linear functionals on $\mathbb{M}_{n_i} \otimes M$, where for each $\varphi: M \to \mathbb{M}_{n_i}$, the corresponding linear functional is given by
\[
\widehat{\varphi}(E_{j,k} \otimes x) = [\varphi(x)]_{j,k},
\]
where $E_{j,k}$ is the standard matrix unit.  Let $A = \mathrm{C}^*(x: x \in F) \subseteq M$, and let $\pi: \mathbb{M}_{n_i} \otimes A \to B(H)$ be the GNS representation on the standard form of $\mathbb{M}_{n_i} \otimes M$.  Since this is faithful on $A$, positive linear functionals on $\mathbb{M}_{n_i} \otimes A$ can be weak-$*$ approximated by those given by vectors, which of course extend to normal functionals on $\mathbb{M}_{n_i} \otimes M$.  Therefore, $\widehat{\varphi}$ can be approximated arbitrarily well on $F$ by normal positive functionals $\widehat{\psi}$ on $M$, which in turn correspond to completely positive maps $M \to \mathbb{M}_{n_i}$.  This implies the claim on approximating CP maps.

Finally, suppose that $\varphi: M \to \mathbb{M}_{n_i}$ is UCP.  Assume without loss of generality that $F$ contains $1$.  Fix $\delta > 0$.  Let $\psi$ be a normal CP map such that $\max_{x \in F} \norm{\varphi(x) - \psi(x)} < \delta$.  Let $\sigma$ be a state on $M$.  Note that $(1 + \delta)^{-1} \psi(1) \leq 1$.  Let
\[
\varphi'(x) = (1 + \delta)^{-1} \psi(x) + \sigma(x)[1 - (1 + \delta)^{-1} \psi(1)].
\]
Then $\varphi'$ is a UCP map, and
\[
\norm{\varphi' - \psi} \leq 2[1 - (1 + \delta)^{-1}].
\]
By choosing $\delta$ small enough, we can arrange that $\max_{x \in F} \norm{\varphi'(x) - \varphi(x)} < \varepsilon$.
\end{proof}

\begin{proof}[Proof of Thereom \ref{thm: main 2}]
    Let $G = SL_3(\Z)$, $G_n = SL_3(\Z / n\Z)$, and $\pi_n: G \to G_n$ be the natural quotient map.  Since $L(G_n)$ is finite-dimensional, there is a trace-preserving embedding $\iota_n: L(G_n) \to M$.  Let $\iota: G \to M^\cU$ be given by
    \begin{equation*}
        \iota(g) = [\iota_n(\pi_n(g))]_n,
    \end{equation*}
    which yields a group homomorphism of $G$ into the unitary group such that
    \begin{equation*}
        \tau(\iota(g)) = \lim_{n \to \cU} \tau_{L(G_n)}(\pi_n(g)) = \tau_{L(G)}(g).
    \end{equation*}
    Therefore, $\iota$ extends to a trace-preserving embedding of $L(G)$ into $M^\cU$.
	
	Using the downward L\"owenheim--Skolem theorem, let $N$ be a separable elementary submodel of $M^\cU$ that contains $\Delta_M(M)$ as well as $\iota(L(G))$.  Of course, $N$ is elementarily equivalent to $M^\cU$ and $M$.  Consider the elementary embedding $\varphi: N \to N^{\cU}$ given as the composition of $N \subset M^\cU$ and then $M^\cU \subset N^\cU$.
    
    Suppose for contradiction that $\varphi$ admits a UCP lifting $(\varphi_n)_n$. Then for each $n$, we have a UCP map $\psi_n: L(G) \to L(G_n)$ given by
    \begin{equation*}
        \psi_n(x) = E_{\iota_n(L(G_n))}(E_M(\phi_n(\iota(x)))).
    \end{equation*}
    Note that, for any $g \in G$, we have
    \begin{equation*}
        \phi(\iota(g)) = [\phi_n(\iota(g))]_n
    \end{equation*}
    as well as
    \begin{equation*}
        \phi(\iota(g)) = [\iota_n(\pi_n(x))]_n \subset \prod_{n \to \cU} \iota(L(G_n)) \subset M^\cU \subset N^\cU.
    \end{equation*}
    Hence,
    \begin{equation*}
        \phi(\iota(g)) = [E_{\iota_n(L(G_n))}(E_M(\phi_n(\iota(g))))]_n = [\psi_n(g)]_n.
    \end{equation*}

    Thus, $\|\psi_n(g) - \pi_n(g)\|_2 \to 0$ as $n \to \cU$.  Let $g_1$, \dots, $g_m$ be a Kazhdan set for $G$ with constant $K$.  Since $L(G_n)$ is finite-dimensional, by Lemma \ref{lem: normal perturbation}, there exists a normal completely positive map $\psi_n': L(G) \to L(G_n)$ such that $\max_j \|\psi_n(g_j) - \psi_n'(g_j)\|_2 \to 0$ for each $g$.  Hence, for $\cU$-many $n$, we have $\max_j \norm{\psi_n'(g_j) - \pi_n(g_j)}_2 < 1/8K^2$.  By Lemma \ref{lem: property T kernel}, this implies that $\ker(\id) = \{e\}$ is a finite index subgroup in  $\ker(\pi_n)$.  But $\ker(\pi_n)$ being finite and $G_n$ being finite would imply that $G$ is finite, a contradiction.
\end{proof}

\section{Proof of Theorem \ref{thm: main theorem}}

The Schwarz inequality $\varphi(x)^* \varphi(x) \leq \varphi(x^*x)$ is an essential tool for understanding completely positive maps.  Since we want to prove our result more generally for positive maps, we will use the following analog of the Schwarz inequality given by Ozawa, which works for positive maps in general but replaces $x^*x$ with the Jordan product $x^* \circ x$.  While the cited source only proves the result for unital positive maps, the more general version can be proved using the exact same argument.

\begin{lemma}[Corollary 2.5 of \cite{Ozawa04}]\label{lem: schwarz-ineq}
    Let the Jordan product of $a, b \in M$ be denoted by $a \circ b$, i.e., $a \circ b = \frac{ab + ba}{2}$. Then for a contractive positive map $\phi$, $\phi(x)^\ast \circ \phi(x) \leq \phi(x^\ast \circ x)$. Thus, if $\phi$ is a positive map without assuming it is contractive, we still have $\phi(x)^\ast \circ \phi(x) \leq \|\phi\|\phi(x^\ast \circ x)$.
\end{lemma}

We next show that if an automorphism $\alpha$ lifts to a uniformly bounded sequence of positive maps $\varphi_n$, then the lift can be modified to arrange that $\varphi_n$ is also normal.

\begin{lemma}\label{lem: normal-pos}
    Fix a free ultrafilter $\cU$ on $\N$. Let $M_i$ be a sequence of separable tracial von Neumann algebras and $\cM = \prod_{i \to \cU} M_i$. Suppose $\alpha: \cM \to \cM$ is a trace-preserving automorphism that lifts to a uniformly bounded sequence of positive maps $(\phi^i)_i$. Then it can be lifted to a uniformly bounded sequence of normal positive maps as well.
\end{lemma}

\begin{proof}
    Recall that every map between von Neumann algebras decomposes into the sum of its normal and singular parts and that normal and singular parts of positive maps are positive. Thus, we may write $\phi^i = \phi^i_n + \phi^i_s$ where $\phi^i_n$ is normal positive and $\phi^i_s$ is singular positive, i.e., $\varphi \circ \phi^i_s$ is singular for every normal positive linear functional $\varphi$ on $M_i$. Recall also that the norm of a linear functional is the sum of the norm of its normal part and the norm of its singular part. (See section III.2 of \cite{Tak79}.)

    Now, since $\alpha$ is trace-preserving, we have $\|\tau_{M_i} \circ \phi^i - \tau_{M_i}\| \to 0$ as $i \to \cU$. Indeed, otherwise we may pick $x_i \in M_i$ with $\|x_i\|_\infty \leq 1$ and
    \begin{equation*}
        |\tau_\cU(\alpha([x_i]_i)) - \tau_\cU([x_i]_i)| = \lim_{i \to \cU} |\tau_{M_i} \circ \phi^i(x_i) - \tau_{M_i}(x_i)| \neq 0,
    \end{equation*}
    contradicting the assumption that $\alpha$ is trace-preserving. Thus,
    \begin{equation*}
    \begin{split}
        \|\tau_{M_i} \circ \phi^i_s\| &\leq \|\tau_{M_i} \circ \phi^i_n - \tau_{M_i}\| + \|\tau_{M_i} \circ \phi^i_s\|\\
        &= \|\tau_{M_i} \circ \phi^i_n - \tau_{M_i} + \tau_{M_i} \circ \phi^i_s\|\\
        &= \|\tau_{M_i} \circ \phi^i - \tau_{M_i}\|\\
        &\to 0.
    \end{split}
    \end{equation*}
    as $i \to \cU$. Hence, for any sequence $(x_i)_i \in \prod_{i \in \N} M_i$,
    \begin{equation*}
    \begin{split}
        \tau_{M_i}(\phi^i_s(x_i)^\ast\phi^i_x(x_i)) &= \tau_{M_i}(\phi^i_s(x_i)^\ast \circ \phi^i_x(x_i))\\
        &\leq \|\phi^i_s\|\tau_{M_i}(\phi_s^i(x_i^\ast \circ x_i))\\
        &\leq \|\tau_{M_i} \circ \phi^i_s\|\left(\sup_i \|\phi^i\|\right)\left(\sup_i \|x_i\|^2\right)\\
        &\to 0.
    \end{split}
    \end{equation*}
    as $i \to \cU$. Therefore, $(\phi^i_n)_i = (\phi^i - \phi^i_s)_i$ is a lifting of $\alpha$ as well.
\end{proof}

\begin{proof}[Proof of Theorem \ref{thm: main theorem}]
Since $M_n$ are separable, there are only $\fc$ many sequences of normal linear maps $M_n \to M_n$.  Hence, the cardinality of the set of automorphisms that admit bounded positive lifts is at most $\fc$ by the previous Lemma.  On the other hand, recall that by \cite[Theorem 16.4.1]{Farah2019} the ultraproduct $\mathcal{M} = \prod_{n \to \cU} M_n$ is countably saturated (or $\aleph_1$-saturated).  The density character of the ultraproduct is $\fc$ by \cite[Proposition 5.3]{GJ25}.  We assumed the continuum hypothesis $\aleph_1 = \fc$, and the language of tracial von Neumann algebras is separable.    Hence, \cite[Theorem 16.6.3]{Farah2019} implies that $\Aut(\mathcal{M})$ has cardinality $2^{\fc}$.  Since only $\fc$ many automorphisms admit bounded positive lifts, there must be $2^{\fc}$ many that do not.
\end{proof}

\bibliographystyle{amsalpha}
\bibliography{references}

\end{document}